\documentclass[12pt]{amsart}

\usepackage[top=4cm,bottom=4cm,inner=4cm,outer=4cm]{geometry}

\newcommand{\nc}{\newcommand}
\nc{\G}{{\Gamma}} \nc{\BC}{{\mathbb C}} \nc{\BQ}{{\mathbb Q}}
\nc{\BR}{{\mathbb R}} \nc{\BZ}{{\mathbb Z}} \nc{\BP}{{\mathbb P}}
\nc{\BN}{{\mathbb N}} \nc{\BM}{{\mathbb M}}
\nc{\fH}{{\mathbb H}}

\nc{\U}{{\mathcal U}}
\nc{\PS}{{\mbox{PSL}_2(\BZ)}} \nc{\SL}{{\mbox{SL}_2(\BZ)}}
\nc{\SR}{{\mbox{SL}_2(\BR)}} \nc{\PR}{{\mbox{PSL}_2(\BR)}}
\nc{\GL}{{\mbox{GL}}} \nc{\PQ}{{\mbox{PGL}_2^+(\BQ)}}
\nc{\GR}{{\mbox{GL}_2^+(\BR)}} \nc{\PG}{{\mbox{PGL}_2^+(\BR)}}
\nc{\GC}{{\mbox{GL}_2(\BC)}}
\nc{\f}{{\mathcal{F}(\fH)}}
\nc{\Cc}{\widehat{\BC}}
\nc{\e}{{E_{\rho}(\G)}}
\nc{\g}{{\gamma}}
\nc{\vm}{{V_{\rho}(\G)}}
\nc{\oo}{{\mathcal O}}
\nc{\M}{{\mbox{M}}}
\nc{\Om}{{\Omega}}
\nc{\TX}{{\widetilde{X}}}
\nc{\ol}{\overline}
\nc{\cl}{{\mathcal L}}
\nc{\ce}{{\mathcal E}}

\newtheorem{numbered}{}[section]
\newtheorem{thm}[numbered]{Theorem}

\newtheorem{remark}[numbered]{Remark}
\newtheorem{prop}[numbered]{Proposition}
\newtheorem{cor}[numbered]{Corollary}

\numberwithin{equation}{section}

\newcommand{\thmref}[1]{Theorem~\ref{#1}}
\newcommand{\propref}[1]{Proposition~\ref{#1}}
\newcommand{\secref}[1]{\S\ref{#1}}

\begin{document}

\title{Vector-valued automorphic forms and vector bundles}
\author[]{Hicham Saber}\author[]{Abdellah Sebbar}
\address{Department of Mathematics and Statistics, University of Ottawa, Ottawa Ontario K1N 6N5 Canada}
\email{hsabe083@uottawa.ca}
\email{asebbar@uottawa.ca}

\begin{abstract}
While vector-valued automorphic forms can be defined for an arbitrary Fuchsian group $\Gamma$ and an arbitrary representation $R$ of $\Gamma$ in
$\GL(n,\mathbb C)$, their existence has been established in the literature only when restrictions are imposed on both $\Gamma$ and $R$. In this paper, we prove the existence  of $n$ linearly independent
vector-valued automorphic forms for any Fuchsian group $\Gamma$ and any $n$-dimensional complex representation $R$ of $\Gamma$. To this end, we realize these automorphic forms as global sections of a special rank $n$ vector bundle built using  solutions to the Riemann-Hilbert problem over various noncompact Riemann surfaces and  Kodaira's vanishing theorem.
\end{abstract}
\maketitle
\section{Introduction}
Let $\G$ be a Fuchsian group, that is a discrete subgroup of $\PR$ acting (properly and discontinuously) on the Poincar\'e half-plane
\[
\fH=\{z\in\BC:\,\mbox{Im}\;z>0\}\,.
\]
  Let $R$ be an $n$-dimensional complex representation of $\G$, that is a homomorphism
\[
R:\G\longrightarrow \GL(n,\BC)\,.
\]
A vector-valued automorphic form for $\G$ of multiplier $R$ and real weight $k$ is a meromorphic function $F:\fH\longrightarrow \BC^n$ satisfying
\begin{equation}\label{func-eq}
F(\gamma\cdot z)=J_{\gamma}^k(z)\,R(\gamma)\;F(z)\,,\ \ z\in\fH\,,\ \ \gamma\in\G\,,
\end{equation}
where $J_{\gamma}(z)=cz+d$ if $\displaystyle \gamma=\binom{*\ *}{c\ d}$.
In addition, we require that at each cusp of $\G$, $F$ has a  meromorphic behaviour  to be made explicit in the next section. Also, the question of multivaluedness of $(cz+d)^k$ has to be addressed via multiplier systems.

The theory of vector-valued automorphic forms has been around for a long time, first, as a
generalization of the classical theory of scalar automorphic forms,
then as natural objects appearing in mathematics and physics. For instance,
 Selberg suggested   vector-valued forms as a tool to study modular forms for finite index subgroups of the modular group
 \cite{sel} and they also appear  as Jacobi forms \cite{sko,e-z}.
 In physics, they appear as
 characters in rational conformal field theory \cite{dlm,e-s,my}. In the last decade, there has been a growing interest in the study
 of vector-valued forms, and several important results have been obtained by various authors \cite{ban-gan,gan2,kn-mas1,kn-mas2,mar1,mar-mas,mas1}.
 However, most of what is found in the literature  deals only with the typical case of $\G$ being the modular group $\SL$ or its subgroups and the convenient
 condition that $R$ is trivial at the parabolic elements of $\G$ and sometimes diagonalizable.
 %Among the tools used to construct vector-valued automorphic forms in these cases, we find  the Riemann-Hilbert problem used in \cite{gan2}.
 A possible generalization to the case of $\G$ being a Fuchsian group of genus zero was announced by  Gannon in \cite{gan2}. More general results were obtained by Borcherds
  in \cite{bor1,bor2} for the metaplectic group of a general cofinite Fuchsian group of the first kind such that the representation has finite order images at the parabolic elements. Also, in \cite{fr}, a general (finitely generated) Fuchsian group of the first kind is considered but for a unitary representation such that the image of a finite index subgroup of $\G$ consists of elements that are simultaneously diagonalizable.
 However, the mere existence of a nonzero vector-valued automorphic form  beyond these cases is, so far, not established, and simple cases such as the defining representation ($R=Id$), the symmetric power representations, monodromy representations of differential equations, as well as any non-unitary representation,  do not fit in the above categories.

Let $X=X_{\G}$ be the quotient of $\fH$ by $\G$ to which we add the cusps (these are not necessarily of finite number).
The Riemann surface may or may not be compact as $\G$ is of any kind. Let $R$ be an $n-$dimensional representation of $\G$.
  In this paper, we prove the following: To the pair ($\G$, $R$), we associate a vector bundle $\mathcal{E}=\mathcal{E}_{\G,R}$ over $X$ such
\[  \mbox{dim}\,H^0(X,\mathcal{M}_{\mathcal{E}})\geq n,
\]
 where $\mathcal{M}_{\mathcal{E}}$ is the sheaf of meromorphic sections of $\mathcal{E}$. The global sections in $H^0(X,\mathcal{M}_{\mathcal{E}})$,
  when lifted to $\fH$, yield $n$ linearly independent vector-valued automorphic forms of multiplier $R$ with a pole at a prescribed cusp. This pole  is removed after multiplying by an appropriate scalar automorphic form, and the same operation is also used to adjust for the desired weight.

The construction of this vector bundle is as follows: There exists a covering $\,\mathcal{U}=(U_i)$ of $X$, parameterized by the set of elliptic fixed points and cusps, such that on each $U_i$, we construct a holomorphic map $\Psi_i:U_i\longrightarrow \mbox{GL}(n,\BC)$ having $R$ as a factor of automorphy (see \secref{sec3}). This is achieved by solving  the Riemann-Hilbert problem on $U_i$ with monodromy $R$. These maps $\Psi_i$ in $\mbox{GL}(n,\mathcal{O}(U_i))$ have the same automorphic behaviour on the overlap of the $U_i$'s yielding a 1-cocycle
$(F_{ij})\in Z^1(\mathcal{U},\mbox{GL}(n,\mathcal{O}))$. To this cocycle, we associate the rank $n$ holomorphic vector bundle $\mathcal{E}$ over $X$ whose transition functions are the maps $F_{ij}$.
The novelty in our construction of the $n$ linearly independent global sections in $H^0(X,\mathcal{M}_{\mathcal{E}})$ lies in the use of the Kodaira vanishing theorem in the compact case, and the property of coherent sheaves in the case $X$ is not compact as it is a Stein variety.

This paper is organized as follows: In \secref{sec2}, we clarify the notion of meromorphy of vector-valued automorphic forms at the cusps. This is necessary
as this behaviour at the cusps is made explicit in the literature only when the representation is unitary or diagonalizable whence, in this paper, we are dealing with arbitrary representations. We only deal with integral weights as this will only be used for weight zero. We go back to arbitrary real weight later in the paper. In \secref{sec3}, we recall the basic notions of differential equations over Riemann surfaces,
their monorodromy and their fundamental systems of solutions. We also introduce the notion of factors of automorphy  and we explain how to get a solution
to the Riemann-Hilbert problem over the punctured disc. We follow closely the treatment of \cite{of}. In \secref{sec4}, we consider a Fuchsian group $\G$
acting on $\fH$ and the noncompact surfaces $Y'=\fH-E$ and $X'=\G\backslash{(\fH-E})$ where $E$ is the set of elliptic fixed points of $\G$. We have
a Galois covering $\pi:Y'\longrightarrow X'$ whose group of covering transformations is $\G$. We then adapt the solution to the Riemann-Hilbert problem for
noncompact Riemann surfaces given in \cite{of} to our covering $\pi:Y'\longrightarrow X'$. In \secref{sec5}, we construct an open covering $\,\U=(U_i)_{i\geq 0}$ of $X$ such that $U_0=X'$ and
 $U_i$, $i\geq 1$, is a neighborhood of (the class of) an elliptic fixed point or a cusp in $X$. Using the solutions to the various Riemann-Hilbert problems from \secref{sec3} and \secref{sec4},
 we construct a cocycle $(\widetilde{F}_{ij})$ in $Z^1(\U,\GL(n,\oo_{X}))$. % with coefficients in the sheaf $\GL(n,\oo_{X})$.
 In \secref{sec6}, we consider the holomorphic vector bundle $\mathcal E$ whose transition functions are given by the cocycle $(\widetilde{F}_{ij})$, and we use Kodaira's vanishing theorem
 to prove that this vector bundle has $n$ linearly independent global sections. These in turn, when lifted to $\fH$, yield $n$ linearly independent vector-valued automorphic forms for $\G$
 with multiplier $R$. In \secref{sec7} we consider the case $n=1$ and $n=2$ which are  respectively related to the notions of generalized modular forms and to equivariant functions.

\section{Behaviour at the cusps}\label{sec2}
In this section we shall make explicit the behaviour of a meromorphic vector-valued automorphic form  at a cusp.
We will restrict ourselves to the case of integral weights as the main results of this paper deal only with this case.
The more general case of a real weight can be dealt with in the same manner with the help of multiplier systems.

Let $\G$ be a discrete group and $R$ a representation of $\Gamma$ in $\GL(n,\mathbb C)$.
From now on, we will denote the matrix $R(\gamma)$, $\gamma\in\G$, by $R_{\gamma}$.
 Let $F$ be a meromorphic function on $\fH$ satisfying \eqref{func-eq} and define the slash operator $|_k$, $k$ being an integer, by
\begin{equation}\label{slash}
F|_k\gamma=(\mbox{det}\,\gamma)^{\frac{k}{2}}\;J_{\gamma}^{-k}\;F\circ\gamma\ ,\ \,\gamma\in\GR\,.
\end{equation}
If $s$ is a cusp of $\G$ and $\alpha\in\SR$ such that $\alpha \cdot s=\infty$. Then, using the notations of \cite{sh}, we have
\[
\alpha\G\alpha^{-1}\{\pm1\}\,=\,\left\{\pm\binom{1\ \ h}{0\ \ 1}^m\,,\ m\in\BZ\right\}\,,
\]
with $h$ being a positive real number referred to as the cusp width at $s$. If we set  $\displaystyle t_h=\binom{1\ \ h}{0\ \ 1}$
then $t_h=\alpha\gamma_s\alpha^{-1}$ with $\G_s=\langle\gamma_s\rangle$ being the fixing group of the cusp $s$ inside $\G$.
Further, if we define $\tilde{F}:=F|_k\alpha^{-1}$, then we have
\begin{align*}
\tilde{F}|_kt_h\,&=\,F|_k\alpha^{-1}|_kt_h\\
&=\,F|_k\alpha^{-1}t_h\\
&=\,F|_k\gamma_s\alpha^{-1}\\
&=\,(F|_k\gamma_s)|_k\alpha^{-1}\\
&=\,R_{\gamma_s}F|_k\alpha^{-1}\\
&=\,R_{\gamma_s}\tilde{F}\,.
\end{align*}
Let $B_s\in\mathcal{M}(n,\mathbb C)$ such that
\[
R_{\gamma_s}\,=\,\exp(2\pi i B_s)\,.
\]
If we set
\[
\Psi_s(z)\,=\,\exp\left(2\pi i \frac{z}{h}B_s\right)\,,  \ z\in\BC\,,
\]
we have
\begin{align*}
\Psi_s(t_hz)\,&=\,\exp\left(2\pi i\frac{z+h}{h}B_s  \right)\\
&=\,\exp\left(2\pi i \frac{z}{h}B_s+2\pi iB_s\right)\\
&=\,R_{\gamma_s}\Psi_s(z)\\
&=\,\Psi_s(z)R_{\gamma_s}\,.
\end{align*}
Hence,
\[
\Psi_s^{-1}\tilde{F}|_kt_h\,=\,\Psi_s^{-1}\tilde{F}\,.
\]
It follows that that $\Psi_s^{-1}\tilde{F}$ has a Laurent expansion in $q_h:=\exp(2\pi iz/h)$.

We say that $F$ is meromorphic at the cusp $s$ if this Laurent expansion has the form
\[
\sum_{n\geq n_0}\,a_nq_h^n\ ,\ \ 0<|q_h|<r
\]
for some integer $n_0$ and some positive real number $r$. Also, $F$ is said to be holomorphic (resp. cuspidal)  if $n_0\geq 0$ (resp. $n_0>0$).

We  adopt this behaviour at the cusps for all the vector-valued automorphic forms constructed in this paper. In fact, this method generalizes
all other suggested ways in the literature. In particular, there is no need to ask that
the matrices $R_{t_h}$ should be unitary or diagonalizable.

\section{Differential equations and  automorphy}\label{sec3}
In this section we recall some facts on  differential equations on Riemann surfaces and their monodromies.
as well as the behaviour of the solutions under the action of the covering transformations of these Riemann surfaces.
Our main reference is \cite{of}.
Let $X$ be a Riemann surface and $A\in \M(n,\Om(X))$ be a $n\times n$ matrix with coefficients
in the space of holomorphic $1-\,$forms on $X$. In other words, for any local chart $(U,z)$ on $X$, one has $A=Fdz$
 where $F\in\M(n,\oo(U))$, $\oo(U)$ being the set of holomorphic functions on $U$. We  then consider the differential equation
 \begin{equation}\label{eq11}
 dw\,=\,Aw \,,
 \end{equation}
 with the unknown map $w:X\longrightarrow \BC^n$ being a holomorphic function on $X$. Locally on $(U,z)$, this equation becomes
 \[
 \frac{dw}{dz}\,=\,Fw\,.
 \]
 Suppose $X$ is simply connected. For every $x_0\in X$ and $c\in\BC^n$,
 there exists a unique  $w\in\oo(X)^n$ solution to \eqref{eq11} satisfying $w(x_0)=c$.

 For an arbitrary Riemann surface $X$, let $p:\TX\longrightarrow X$ be its universal covering, $x_0\in X$ and $y_0\in\TX$ such that $p(y_0)=x_0$. Then for each $c\in\BC^n$, there exists a unique solution $w\in\oo(\TX)^n$ on $\TX$ to
 the differential equation
 \begin{equation}\label{eq12}
 dw\,=\,(p^*A)w
 \end{equation}
satisfying $w(z_0)=c$. Here $p^*A=A\circ p$.

In general, let $L_A$ be the set of all solutions $w\in\oo(\TX)^n$ to the differential equation \eqref{eq12}.
Then $L_A$ is an $n-$dimensional complex vector space, and a family of solutions $w_1,\ldots,w_n$ form a basis
if, for each $a\in\TX$, the vectors $w_1(a),\ldots,w_n(a)\in\BC^n$ are linearly independent.
In this case, we have an invertible matrix
\begin{equation}\label{eq13}
\Phi\,=\,[w_1,\ldots,w_n]\,\in\,\GL(n,\oo(\TX))
\end{equation}
such that
\begin{equation}\label{eq14}
d\Phi\,=\,(p^*A)\Phi\,.
\end{equation}
The matrix $\Phi$ is called a {\em fundamental system} of solutions of the differential equation $dw=Aw$.

Let $G:=\mbox{Deck}\,(\TX/X))$ be the group of covering transformation of the covering $p:\TX\longrightarrow X$, that is
the set of automorphisms of $\TX$ that preserve the fibers of $p$.
In our case, since the covering is universal, then $G$ is isomorphic to $\pi_1(X)$, the fundamental group of $X$. For $\sigma\in G$,
set $\sigma\Phi=\Phi\circ\sigma^{-1}$. Then $\sigma \Phi$ satisfies the equation \eqref{eq14}.
In other words, we have
$d(\sigma \Phi)=(p^*A)\sigma\Phi$ yielding another fundamental system of solutions. Therefore, there exists a constant
matrix $R_{\sigma}\in\GL(n,\BC)$ such that
\begin{equation}\label{eq15}
\sigma\,\Phi\,=\,\Phi\,R_{\sigma}\,.
\end{equation}
If $\tau\in G$ is another covering transformation, it is clear that
\[
R_{\sigma\tau}\,=\,R_{\sigma}\,R_{\tau}\,.
\]
Therefore, we have a representation $R$ of $G=\pi_1(X)$ in $\GL(n,\BC)$.

In general, let $p:Y\longrightarrow X$ be a holomorphic unbranched covering  of Riemann surfaces and let $G$ be the group of covering transformations. A holomorphic mapping $\Psi:Y\longrightarrow \GL(n,\BC)$ is called automorphic with constant factors of automorphy $R_{\sigma}\in\GL(n,\BC)$ if
\[
\sigma \Psi=\Psi R_{\sigma}\,\mbox{ for all }\, \sigma\in G\,.
\]
This defines
a representation $R$ of $G$ in $\GL(n,\BC)$. In particular, according to \eqref{eq15}, the fundamental system of solutions $\Phi$ is thus automorphic with $R_{\sigma}$ as factors of automorphy.

Conversely, suppose we are given a representation $R$ of $G=\pi_1(X)$ in $\GL(n,\BC)$
and a holomorphic mapping $\Phi:\TX\longrightarrow \GL(n,\BC)$ satisfying $\sigma\Phi =\Phi R_{\sigma}$ for all $\sigma\in G$. The matrix $d\Phi\cdot\Phi^{-1}\in \M(n,\Om(\TX))$
is invariant under covering transformations:
$$
\sigma(d\Phi\cdot\Phi^{-1})=(d\Phi\cdot R_{\sigma})(\Phi R_{\sigma})^{-1}=d\Phi\cdot\Phi^{-1}.
$$
Therefore, there exists a matrix $A\in \M(n,\Omega(X))$ such that $p^* A=d\Phi\cdot \Phi^{-1}$.
Moreover, $\Phi$ is a fundamental system of solutions of the differential
equation $dw=Aw$.

In the following theorem, we will see that for a punctured disc and for any
 invertible matrix $R$, one can always find a differential equation
whose solution has $R$ as its factor of automorphy. Recall that on $\TX$, the logarithm function is well defined, that is,
 there exists a holomorphic function
 \[
 \log:\TX\longrightarrow \BC
 \]
 such that $\exp\circ\log=p$ where $p:\TX\longrightarrow X$ is the universal covering of $X$.
\begin{thm}\label{punc-disc}~\cite{of}
Suppose $R\in\GL(n,\BC)$ and $B\in\M(n,\BC)$ such that
\[
\exp(2\pi iB)=R\,.
\]
 Then
on $X:=\{z\in\BC\,:\ 0<|z|<R\}$,  the holomorphic matrix $\Phi_0(z)=\exp(B\log z)$ is
a fundamental system of solutions of
\[
w'=\frac{1}{z}\,Bw
\]
on the universal covering $p:\TX\longrightarrow X$ which has $R$ as its
factor of automorphy, i.e.
\[
\sigma\Phi_0\,=\,\Phi_0R,
\]
$\sigma$ being a generator of $\mbox{Deck}\,(\TX/X)$ such that
\[
\sigma\log\,=\,2\pi i+\log.
\]
\end{thm}
\begin{proof}
This can be shown by straightforward calculations: \\
First, it is clear
that $\displaystyle \Phi_0'=\frac{1}{z}B\Phi_0$ and
\[
\sigma\Phi_0=\sigma\exp(B\log)=\exp(B\sigma\log)=\exp(B(log+2\pi i))
\]
\[=
\exp(B\log)\exp(2\pi iB)=\Phi_0R.
\]
Notice that given the choice of $\sigma$, the group %of covering transformations
$\mbox{Deck}\,(\TX/X)$ is given by the monodromy group
$\{\sigma^n\,,\ n\in\BZ\}$.
\end{proof}

\section{The case of non-compact Riemann surfaces}\label{sec4}

Let $\G$ be a Fuchsian group,  and let $E$ be the set of its
elliptic fixed points and $C$ the set of its cusps in $\BR\cup\{\infty\}$. Define
\[
X=X(\G)\,:=\,\G\backslash(\fH\cup C)\,\ ,\ \ X'=X'(\G):=\G\backslash(\fH- E)\,.
\]
Notice that both $\G\backslash E$ and $\G\backslash C$ are discrete closed sets, that are finite if $\G$ is a Fuchsian group of the first kind. If we
set $Y'=\fH-E$, then we have an unbranched covering map
\[
\pi:Y'\longrightarrow X'\,.
\]
Since $\G$ acts properly and discontinuously on $Y'$, then the group of
covering transformations is
\[
\mbox{Deck}\,(Y'/X')\,=\,\G.
\]
Moreover, this is a Galois covering in the sense that for all $y_1$ and $y_2$
in $Y'$ with $\pi(y_1)=\pi(y_2)$, there exists $\sigma\in\G$ such
$\sigma(y_1)=y_2$.

The following theorem is a key tool in the solution of the Riemann-Hilbert
for noncompact Riemann surfaces. For the sake of completeness, we
include the proof  from \cite{of} adapted to our Riemann surfaces.

\begin{thm} \label{rh}
Let $\G$ be a Fuchsian group and define
 the noncompact Riemann surfaces $Y'$ and $X'$ as above. If $R$ is a complex
 representation of $\G$ in $\GL(n,\BC)$, then there exists a holomorphic map
 $\Phi:Y'\longrightarrow\GL(n,\BC)$ having $R_{\sigma}$ as factors of
 automorphy. That is, for all $\sigma\in \G$, we have
 \[
 \sigma\Phi\,=\,\Phi R_{\sigma}\,.
 \]
 \end{thm}
 \begin{proof}
 Since $\pi:Y'\longrightarrow X'$ is an unbranched Galois covering,
 there exists an open covering $\U=(U_i)_{i\in I}$ of $X'$
 and homeomorphisms (also known as $\G-$charts)
 \[
 \phi_i=(\pi,\eta_i):\pi^{-1}(U_i)\longrightarrow U_i\times \G.
 \]
 The mapping $\phi_i$ is fiber-preserving and is compatible with the action
 of $\G$ in the sense that $\phi_i(y)=(x,\sigma)$ implies
 $\phi(\tau y)=(x,\tau\sigma)$. In other words, the mapping
 $\eta_i:\pi^{-1}(U_i)\longrightarrow \G$ satisfies $\eta_i(\tau y)=\tau\eta_i(y)$
 for all $y\in\pi^{-1}(U_i)$ and $\tau\in\G$.

 On $Y_i=\pi^{-1}(U_i)$ define $\Psi_i:Y_i\longrightarrow \GL(n,\BC)$ by
 \[
 \Psi_i(y)\,=\,R_{\eta_i(y)^{-1}}\ ,\ \ y\in Y_i.
 \]
Since $\Psi_i$ is locally constant, it is holomorphic. Now, if $y\in Y_i$ and
$\sigma\in\G$, then
\[
\sigma\Psi_i(y)=\Psi_i(\sigma^{-1}y)=R_{\eta_i(\sigma^{-1}y)^{-1}}=
R_{\eta_i(y)^{-1}\sigma}=R_{\eta_i(y)^{-1}}R_{\sigma}=\Psi_i(y)R_{\sigma}.
\]
Therefore, $\Psi_i$ is automorphic on $Y_i$ with factors of automorphy
$R_{\sigma}$.
If we set
\[
F_{ij}=\Psi_i\Psi_j^{-1}\in\GL(n,\oo(Y_i\cap Y_j))\,,
\]
 then for all $y\in Y_i\cap Y_j$ we have $\sigma F_{ij}(y)=F_{ij}(y)$, that is,
 the $F_{ij}$ is invariant under covering transformations and hence
 may be considered as an element of $\GL(n,\oo(U_i\cap U_j))$. Therefore, we
 have a cocycle
 \[
 (F_{ij})\in Z^1(\U,\GL(n,\oo))\,.
 \]
As $X'$ is a  noncompact Riemann surface, we have \cite{of,gun}
\[
H^1(X',\GL(n,\oo))\,=\,0\,.
\]
Therefore, there exist elements $F_i\in\GL(n,\oo(U_I))$ such that
\[
F_{ij}\,=\,F_iF_j^{-1}\  \mbox{ on }\ U_i\cap U_j\,.
\]
We now look at the $F_i$ as elements of $\GL(n,\oo(Y_i))$ that are invariant
under covering transformations and set
\[
\Phi_i=F_i^{-1}\Psi_i\in\GL(n,\oo(Y_i))\,.
\]
Then, for every $\sigma\in \G$, we have
\[
\sigma\Phi_i=F_i^{-1}\sigma \Psi_i=F_i^{-1}\Psi_i R_{\sigma}=\Phi_i R_{\sigma}.
\]
Moreover, on $Y_i\cap Y_j$ we have
\[
\Phi_i^{-1}\Phi_j=\Psi_i^{-1}F_iF_j^{-1}\Psi_j=\Psi_i^{-1}F_{ij}\Psi_j=
\Psi_i^{-1}\Psi_i\Psi_j^{-1}\Psi_j=1.
\]
Thus, the $\Phi_i$'s define a global function $\Phi\in\GL(n,\oo(Y'))$ with
\[
\sigma\Phi\,=\,\Phi R_\sigma\,,\ \sigma\in\G\,.
\]

 \end{proof}

\section{A cocycle with coefficients in the sheaf $\GL(n,\oo_{X})$}\label{sec5}

Throughout this section, $\G$ is a Fuchsian group  and $R$ is
a representation of $\G$ in $\GL(n,\BC)$. Recall the unbranched covering $\pi:Y'\longrightarrow X'$ from the previous section.
We will extend the previous constructions to the elliptic
points and the cusps of $\G$. Write
\[
\G\backslash (C\cup E)\,=\,\{\ol{a}_i\}_{i\geq 1}
\]
for the discrete closed set of classes of cusps and elliptic fixed points in $X$. In particular, they correspond to inequivalent
points $a_i$ in $\fH\cup\BR$. For each $\ol{a}_i$ we choose a neighborhood in $X^*(\G)$ in the following way:

If $\ol{a}_i$ is a cusp, let $U_i$ be a neighborhood of $\ol{a}_i$ in $X$ given by
\[
U_i\,=\,(\G_{a_i}\backslash D_{a_i})\cup \{\ol{a}_i\},
\]
where $\G_{a_i}$ is the stabilizer of $a_i$ in $\G$ and $D_{a_i}$ is a horocycle in $\fH$
tangent at $a_i$ \cite{sh} (if $a_i=\infty$, this horocycle is a half plane). Here $\G_{a_i}$ is infinite cyclic.

If $\ol{a}_i$ is an elliptic fixed point, then $U_i$ is given by
\[
U_i\,=\,\G_{a_i}\backslash D_{a_i}\,,
\]
where $D_{a_i}$ is an open disc in $\fH$ centered at $a_i$. Here $\G_{a_i}$ is a finite cyclic group.

Further, these neighborhoods can be taken such that for $\ol{a}_i\neq\ol{a}_j$, we have $U_i\cap U_j=\emptyset$.
Also, for each  $\ol{a}_i$ choose a chart $z$ for $X$ such
that $z(U_i)=D$ and $z(a_i)=0$ where $D$ is the unit disc.

We now set $U_0=X'$, $V_0=\fH-E=\pi^{-1}(X')=Y'$. Then $\U=(U_i)_{i\geq 0}$ is an open covering
of $X$. \thmref{rh} provides us with a $R-$automotphic map $\Phi:V_0=Y'\longrightarrow \GL(n,\BC)$ and we set $\Psi_0=\Phi$.

\begin{prop}\label{cusp}
If $\ol{a}_i$ is a cusp and $V_i=\pi^{-1}(U_i-\{\ol{a}_i\})$, there exists a holomorphic map
\[
\Psi_i: \,V_i\longrightarrow \GL(n,\BC)
\]
having the same automorphic behaviour as $\Psi_0|{V_i}$.
%which is $R-$automorphic on connected components.
\end{prop}
\begin{proof}
We have
\[
V_i\,=\,\G D_{a_i}\,=\,\bigsqcup_{\ol{\gamma}\in\G/\G_{a_i}}{\gamma}D_{a_i}\,=\,
\bigsqcup_{\ol{\gamma}\in\G/\G_{a_i}}D_{\gamma a_i}\,.
\]
Here we have a disjoint union of connected components $Z_{\gamma}={\gamma}D_{a_i}$ of $V_i$
and $\gamma$ is a chosen representative of $\ol{\gamma}$ in $\G/\G_{a_i}$. Then
\[
Z_{\gamma}\longrightarrow U_i-\{\ol{a}_i\}
\]
is a universal covering with the covering transformations given by the
fundamental group
\[
\pi_1(U_i-\{\ol{a}_i)\}\,=\,\G_{\gamma a_i}\,=\,\langle\gamma_i\rangle\,,
\]
where $\gamma_i$ is a generator of $\G_{\gamma a_i}$. In case the cusp $ \ol{a}_i$ is  at $\infty$, then
this covering is isomorphic to the universal covering
\[
\fH\longrightarrow \fH/\langle z\mapsto z+h\rangle\cong D^*,
\]
where the chart is given by $q=\exp (2\pi i z/h)$, $h$ being the cusp width at $\infty$ and $D^*=\{q\in{\mathbb C} \mid 0<|q|<1\}$.

Using the chart $z$ (or $q$ when $\infty$ is a cusp for $\G$)  and
\thmref{punc-disc},
let $B_i\in\M(n,\BC)$ such that
\[
\exp(2\pi iB_i)\,=\, R_{\gamma_i}\,.
\]
Then the holomorphic function $\Psi_{\gamma_i}(z)=\exp(B_i\log z)$ on $Z_{\gamma}$
satisfies
\[
\gamma_i\Psi_{\gamma_i}\,=\,\Psi_{\gamma_i}R_{\gamma_i}\,,
\]
which is the same automorphic behaviour as the one for $\Psi_0|_{Z_{\gamma}}$ as the monodromy group
of the differential equation given by $\Psi_0|_{Z_{\gamma}}$ is $\langle \gamma_i\rangle$.
We define $\Psi_i$ on $V_i$ by
\[
\Psi_i|_{Z_{\gamma}}=\Psi_{\gamma_i}\,.
\]
\end{proof}
\begin{prop}\label{elliptic}
If $\ol{a}_i$ is an elliptic fixed point and $V_i=\pi^{-1}(U_i-\{\ol{a}_i\})$,
there exists a holomorphic map
\[
\Psi_i: \,V_i\longrightarrow \GL(n,\BC)
\]
having the same automorphic behaviour as $\Psi_0|{V_i}$. Moreover, $\Psi_i$ is meromorphic at each point of $\pi^{-1}\{\ol{a}_i\}$.
%which is $R-$automorphic on connected components.
\end{prop}
\begin{proof}
If $\ol{a}_i$ is the class of an elliptic fixed point $a_i$ in $\fH$, then
\[
U_i-\{\ol{a}_i\}\,=\,\G_{a_i}\backslash D^*_{a_i}\,,
\]
where $D^*_{a_i}=D_{a_i}-\{a_i\}$ is the punctured disc. Here, the stabilizer
$\G_{a_i}=\langle\gamma_i\rangle$ of finite order $k_i$.
Furthermore, we have
\[
V_i\,=\,\bigsqcup_{\ol{\gamma}\in\G/\G_{a_i}}{\gamma}D_{a_i}^*\,=\,
\bigsqcup_{\ol{\gamma}\in\G/\G_{a_i}}Z_{\gamma}\,,
\]
where the $Z_{\gamma}=\gamma D_{a_i}^*$ are the connected components of $V_i$.
In the meantime, the unbranched covering $Z_{\gamma}\longrightarrow D_{a_i}^*$
is isomorphic to the covering
\begin{align*}
D^*&\longrightarrow D^*\\
z&\longmapsto z^{k_i}
\end{align*}
for which the covering transformations are given by the group
$\langle\gamma'_i:z\longmapsto \zeta_i z\rangle$ where $\zeta_i$ is a primitive $k_i-$th root of unity.
Furthermore,
since $R_{\gamma_i}^{k_i}=1$,  then $R_{\gamma_i}$ is diagonalizable and for
some $P\in\GL(n,\BC)$ we have
\[
PR_{\gamma_i}P^{-1}\,=\,\mbox{diag}\,(\lambda_1,\ldots,\lambda_n)\ ,\quad
\lambda_j^{k_i}=1\,.
\]
Write $\lambda_j=\zeta_i^{\alpha_j}$ and  for $z\in D^*$ set
\[
\Psi'_{\gamma_i}(z)\,=\,\mbox{diag}\,(z^{-\alpha_1},\ldots,z^{-\alpha_n})\,.
\]
We have
\begin{align*}
\gamma'_i\Psi'_{\gamma_i}(z)\,&=\,\Psi'(\zeta_i^{-1}z)\\
&=\,\mbox{diag}\,
(\zeta_i^{\alpha_1}z^{-\alpha_1},\ldots,\zeta_i^{\alpha_n}z^{-\alpha_n})\\
&=\,
\mbox{diag}\,
(\lambda_1 z^{-\alpha_1},\ldots,\lambda_n z^{-\alpha_n})\\
&=\,
\Psi'_{\gamma_i}(z)PR_{\gamma_i}P^{-1}\,.
\end{align*}
Therefore,
\[
\gamma'_i \Psi'_{\gamma_i}P\,=\,\Psi'_{\gamma_i}PR_{\gamma_i}\,.
\]
If $\Psi_{\gamma_i}=\Psi'_{\gamma_i}P$, then going back to the covering
$Z_{\gamma}\longrightarrow D_{a_i}^*$ we get
\[
\gamma_i\Psi_{\gamma_i}\,=\,\Psi_{\gamma_i}R_{\gamma_i}\,.
\]
We now define $\Psi_i:\, V_i\longrightarrow \GL(n,\BC)$
 by $\Psi_i|_{Z_{\gamma}}=\Psi_{\gamma_i}$. It is clear that $\Psi_i|_{Z_{\gamma}}$ and $\Psi_0|_{Z_{\gamma}}$ have the same automorphic behaviour under $\langle \gamma_i\rangle$.
\end{proof}
In summary, we have a covering $\U=(U_i)_{i\geq 0}$ of $X$ such that on $V_0=\pi^{-1}(U_0)$ we have a $R-$automotphic
map $\Psi_0$ (\thmref{rh}), and on $V_i=\pi^{-1}(U_i-\{a_i\})$, $i\geq 1$, we have  a holomorphic map
$\Psi_i$ with the same automorphic behaviour as $\Psi_0|_{V_i}$ (\propref{cusp} and \propref{elliptic}). We now define the cocycle
\[
F_{ij}=\Psi_i\Psi_j^{-1}\in \GL(n,\oo(V_i\cap V_j))\ \,\mbox{if }\ i\neq j\ \,\mbox{and }\ F_{ii}=\mbox{id}\,.
\]
\begin{thm}\label{cocycle}
The cocycle $(F_{ij})$ defines a cocycle  $(\widetilde{F}_{ij})$ in $Z^1(\U,\GL(n,\oo))$ such that $F_{ij}=\pi^*\widetilde{F}_{ij}$.
\end{thm}
\begin{proof}
We need to prove that the $F_{ij}$'s can be considered as elements of $\GL(n,\oo(U_i\cap U_j))$. By construction,
for $i\neq 0$ and $j\neq 0$ we have
\[
U_i\cap U_j=V_i\cap V_j=\emptyset\ ,\, \ U_0\cap U_i=U_i-\{a_i\}\ ,\, \ V_0\cap V_i=V_i\,.
\]
 We only need to prove that $\Psi_0\Psi_i^{-1}\in\GL(n,\oo(V_i))$ defines an element of
$\GL(n,\oo(U_i-\{a_i\}))$. Using \propref{cusp} and \propref{elliptic},  $\Psi_i$ and $\Psi_0$ have
the same automorphic behaviour  under Deck$\,(Z_{\gamma}/U_i-\{a_i\})$ on each connected component $Z_{\gamma}$ of
$V_i$. Therefore, $\Psi_0\Psi_i^{-1}$
is invariant under the action of Deck$\,(Z_{\gamma}/U_i-\{a_i\})$, and so descends to $\GL(n,\oo(U_0\cap U_i))$.
Therefore, there exists  an element $(\widetilde{F}_{ij})\in Z^1(\U,\GL(n,\oo))$ such that $F_{ij}=\pi^*\widetilde{F}_{ij}$.

\end{proof}

\section{Vector-valued automorphic forms as sections of a vector bundle}\label{sec6}

In the previous section, we have constructed an open covering $\U=(U_i)_{i\geq 0}$ of the  Riemann surface $X=X(\G)$
and a cocycle $(\widetilde{F}_{ij})\in Z^1(\U,\GL(n,\oo))$. There exists a holomorphic vector bundle $p:\ce\longrightarrow X$
of rank $n$ whose transition functions are the $(\widetilde{F}_{ij})$ ~\cite{of}, \cite{g-h}. We now choose a positive line
bundle $\cl\longrightarrow X$ which we may take to be one whose associated divisor is
$\cl=[P]$ where $P$ is an arbitrary fixed point on $X$. As $\G$ is a Fuchsian group of any kind, then $X$ may or may not be compact.
We first suppose that $X$ is compact. Then, using Kodaira's vanishing theorem, there exists an integer
$\mu_0\geq 0$ such that for all $\mu\geq\mu_0$ we have
\begin{equation}\label{kod}
H^q(X,\oo(\cl^{\mu}\otimes \ce))\,=\, 0\ ,\ \, q>0\,.
\end{equation}
For convenience, we will adopt the additive notation  $\mu[P]+\ce$ for divisors instead of $\cl^{\mu}\otimes \ce$.
\begin{prop}\label{sections1}
For each $\mu\geq\mu_0$, there are $n$ linearly independent holomorphic sections of $(\mu+1)[P]+\ce$.
\end{prop}
\begin{proof}
We have an exact sequence
\[
0\rightarrow H^0(X,\oo(\mu[P]+\ce))\rightarrow H^0(X,\oo((\mu+1)[P]+\ce))\rightarrow ((\mu+1)[P]+\ce)_P\rightarrow 0
\]
where
\[
H^0(X,\oo((\mu+1)[P]+\ce))\rightarrow ((\mu+1)[P]+\ce)_P
\]
is the evaluation map and $((\mu+1)[P]+\ce)_P\cong \BC^n$ is the fiber over $P$. This yields a long exact sequence
\[
\dots\rightarrow H^0(X,\oo((\mu+1)[P]+\ce))\rightarrow \BC^n\rightarrow H^1(X,\oo(\mu[P]+\ce))\ldots .
\]
Using \eqref{kod}, we have $H^1(X,\oo(\mu[P]+\ce))=0$. Therefore, the evaluation map at $P$
\[
H^0(X,\oo((\mu+1)[P]+\ce))\rightarrow \BC^n
\]
is surjective. We choose sections $\sigma_i$ in $H^0(X,\oo((\mu+1)[P]+\ce))$ such that
\[
\sigma_i(P)\,=\,e_i\,, \  1\leq i\leq n\,,
\]
where the $e_i$'s are the canonical basis vectors of $\BC^n$. This yields $n$ linearly independent sections of $(\mu+1)[P]+\ce$.
\end{proof}

A section $\sigma_i\in H^0(X,\oo((\mu+1)P+\ce))$ can be viewed as a section in  $\sigma_i\in H^0(X-\{P\}, \oo(E))$ having
a pole of order at most $\mu+1$. Therefore, we have $n$ linearly independent meromorphic sections of $\ce$ with a pole of order at most $\mu+1$ at $P$. This is particularly true for $\mu=\mu_0$.

If $X$ is not compact, then it is a Stein variety \cite{of}, and since $\oo(\ce)$ is a coherent sheaf over $X$, then one can take $\mu_0=0$ \cite{gun}, and we reach the same conclusion as above with a pole of order $1$.

We now set
\[
\tilde{\Phi}\,=\, [\sigma_1,\ldots,\sigma_n]\ ,\ \, \sigma_i\in H^0(X,{\mathcal M}_{\ce})\,\mbox{ as above},
\]
 and let
\[
A_i=\Psi_i^{-1}\pi^*\tilde{\Phi}=\Psi_i^{-1}\tilde{\Phi}\circ \pi\ \mbox{on }\, V_i\,,
\]
where $\tilde{\Phi}_i=\tilde{\Phi}|_{U_i}$.
Then for all $\gamma\in\G$,
\[
\gamma A_i=R_{\gamma}A_i\, .
\]
Moreover, $A_i$ has a pole of order at most $\mu+1$ at $\pi^{-1}(P)$ and it is holomorphic
elsewhere including at cusps and elliptic fixed points away from $P$.  On $V_i\cap V_j$, as the $F_{ij}$
are the transition functions of the vector bundle $\ce$, we have
\begin{align*}
A_i\,&=\,\Psi_i^{-1}\pi^*\tilde{\Phi}_i\\
&=\,\Psi_i^{-1}\pi^*(\widetilde{F}_{ij}\tilde{\Phi}_j)\\
&=\,\Psi_i^{-1}\pi^*(\widetilde{F}_{ij})\pi^*(\tilde{\Phi}_j)\\
&=\,\Psi_i^{-1}F_{ij}\pi^*(\tilde{\Phi}_j)\\
&=\,\Psi_i^{-1}\Psi_i\Psi_j^{-1}\pi^*\tilde{\Phi}_j\\
&=\,A_j\,.
\end{align*}
Therefore, $A=(A_i)$ is globally well defined. If we initially place the point $P$ at a cusp, then each $A_i$ is
holomorphic on $\fH$ and so does $A$ defined by $A|_{V_i}=A_i$. Each column of $A$ provides a nonzero vector-valued automorphic form
 for $\G$ of multiplier $R$ and we obtain the following:
\begin{thm}\label{main}
Let $\G$ be a Fuchsian group, and let $R$ be a representation of $\G$ in $\GL(n,\BC)$.
There exist $n$ linearly independent vector-valued automorphic forms  of weight 0 for the group $\G$ and multiplier $R$ which are meromorphic at a cusp and holomorphic elsewhere.
\end{thm}\qed

For an arbitrary real weight, we obtain existence results by multiplying by appropriate scalar automorphic forms for $\G$ provided they exist.
\begin{cor} Let $F_i$, $1\leq i\leq n$, be the linearly independent vector-valued automorphic forms constructed in \thmref{main}.
\begin{enumerate}
\item If there is a scalar automorphic form $f$ of weight $k$ for $\G$, then  $\{fF_1,\ldots fF_n\}$ are linearly independent vector-valued automorphic forms of weight $k$ for $\G$ of multiplier $R$.
\item If there exists a holomorphic automorphic form $f$ vanishing at the cusp where the $F_i$'s have a pole, then, for some  integer $r\geq 0$,
 $\{f^rF_1,\ldots f^rF_n\}$ are linearly independent holomorphic vector-valued automorphic forms for $\G$ of multiplier $R$.
\end{enumerate}
\end{cor}
\qed

\section{Special cases}\label{sec7}
The notion of generalized modular forms was introduced by  Knopp and
Mason in \cite{kn-mas0}, and was further developed in  \cite{ko-mas,kn-mas01,raji}. A meromorphic function
$f$ on $\fH$ is called a generalized modular form of weight $k$ for a modular subgroup $\G$
of $\SL$ if for all $\alpha\in \G$ and $z\in\fH$, we have
\[
f|_k(z) =\mu(\alpha) f(z),
\]
where $\mu : \G\longrightarrow \mathbb C$ is a character. In addition,
 $f$ should be meromorphic at the cusps.

The generalized modular forms look like classical modular forms with
a multiplier system except for the fact that the character $\mu$ need not be unitary. Only the parabolic
generalized modular forms, that is when $\mu$ is trivial on the parabolic elements of $\G$, were subject to study in the above-mentioned references (among others). This is mainly due to existence problems when $\mu$ is an arbitrary character of $\G$. However, \thmref{main} in dimension 1 guarantees the existence of generalized modular forms of weight 0 for an arbitrary character of an arbitrary discrete group $\G$. Moreover, the weight can be nonzero provided a scalar automorphic form with the needed weight exists.

In the 2-dmensional case, we have a nice application to the so-called equivariant functions on the upper-half plane. These were introduced and studied in \cite{ss1,sb1,sb2} and they are defined as meromorphic functions
$h\,:\,\fH\longrightarrow \BC$ that commutes with the action of $\G$, that is,  for all $z\in\fH$ and $\gamma\in\G$, we have
\[
h(\gamma\cdot z)\,=\,\gamma\cdot h(z)\,,
\]
where the action on both sides is by linear fractional transformations. A trivial example is given by $h_0(z)=z$. Nontrivial examples are constructed from
automorphic forms. Indeed, if $f$ is a weight $k$ automorphic form for $\G$, then
\[
h_f(z)\,=\,z\,+\,k\;\frac{f(z)}{f'(z)}
\]
is equivariant for $\G$. These are referred to as the rational equivariant functions \cite{sb2}, and they do not account for all the equivariant functions.
In fact, it is shown in \cite{sb1} that the set of equivariant functions  has a structure of an infinite dimensional vector space.
Moreover, each equivariant function can be viewed as a section of the canonical bundle of $\G\backslash\fH$.

The equivariant functions can be generalized in the following way: Let $\rho$ be a two-dimensional representation of a discrete group $\G$. A meromorphic
function $h$ on $\fH$ is called  $\rho$-equivariant  if for all $z\in\fH$ and $\gamma\in\fH$, we have
\[
h(\gamma\cdot z)\,=\,\rho(\gamma)\cdot f(z)\,.
\]
 In particular, an equivariant function corresponds
 to the trivial representation $\rho=\mbox{Id}$.
 Using the monodromy of differential equations and the Schwarz derivative we have
 \begin{thm}\cite{ss2} Let $\G$ be a Fuchsian group of the first kind and $\rho$ a 2-dimensional complex representation of $\G$.
 A meromorphic function $h$ on $\fH$ is $\rho-$equivariant for $\G$ if and only if $\displaystyle h\,=\,f_1(z)/f_2(z)$ where
 $\displaystyle F=\binom{f_1}{f_2}$ is a vector-valued automorphic form of multiplier $\rho$ and an arbitrary weight.
 \end{thm}
  We should mention that the assumption that the Fuchsian group $\G$ is of the first kind is needed only in the "only if" part of the above theorem, which is the most difficult part. The "if" part is straightforward and does not require any assumption on $\G$.
  In the context of this paper, the existence of vector-valued automorphic forms yields the previously
unknown result:
\begin{thm} If $\G$ is a Fuchsian group, then $\rho$-equivariant functions exist for every  representation $\rho$ of $\G$ in $\GL(2,\BC)$.
\end{thm}

\begin{remark}{\rm
 The study of the dimensions and the structures of the
spaces of these vector-valued automorphic forms will appear in a forthcoming work.}
\end{remark}

\end{document}